\documentclass[11pt,a4paper, final]{article}
\usepackage{template}
\usepackage{calc}
\usepackage{pgfplots}
\usepackage{upgreek}

\usepackage[normalem]{ulem}

\DeclareMathOperator\Circ{Circ}
\DeclareMathOperator\E{E}
\let\P\relax
\DeclareMathOperator\P{P}

\newcommand{\Z}{\mathbb{Z}}
\newcommand{\R}{\mathbb{R}}
\newcommand{\N}{\mathbb{N}}

\usetikzlibrary{patterns}

\begin{document}

\title{No infinite clusters at criticality for two-dimensional dependent percolation under abstract hypotheses}

\date{\today}
\author{
  Rangel Baldasso
  \thanks{Email: \ \texttt{rangel@puc-rio.br}; \ Department of Mathematics, PUC-Rio, Rua Marqu\^es de S\~ao Vicente 225, G\'avea, 22451-900 Rio de Janeiro, RJ - Brazil.}
}

\maketitle

\begin{abstract}
We establish a general criterion ensuring the absence of infinite connected components at criticality for a broad class of two-dimensional dependent percolation models. Our assumptions include monotonicity, stationarity and ergodicity, positive association, a quantitative decoupling condition, continuity in the percolation parameter, and trivial behavior at the extremes. As application, we show how our criterion recovers known results for classical dependent percolation models, such as Gaussian level-set percolation and Boolean percolation.

\medskip


\end{abstract}

\section{Introduction}

The study of critical percolation in two dimensions has a long history, beginning with Harris' 1960 work~\cite{harris1960}, proving that $p_c \geq \frac{1}{2}$ for Bernoulli bond percolation on the square lattice. Later on, Kesten~\cite{kesten1980} identified the critical threshold $p_{c} = \frac{1}{2}$ via delicate analysis of crossing probabilities.

The Russo--Seymour--Welsh (RSW) theory, developed independently by Russo~\cite{russo} and by Seymour and Welsh~\cite{seymour_welsh} in the late 70s, provides the fundamental estimates on crossing probabilities that underlie much of modern two-dimensional percolation. At first, these estimates yielded scale-invariant bounds at criticality and are by now a key tool in the analysis of planar models. Together with the FKG inequality, they form the basis of many proofs concerning the geometry of critical clusters.

The absence of an infinite cluster at criticality extends far beyond independent bond percolation on the square lattice. Several analogous results have been obtained through variants of the RSW theory and planar duality (for instance, the Boolean percolation~\cite{mr, att, att2}, Voronoi percolation~\cite{br}, confetti percolation~\cite{muller}). Modern treatments of these facts emphasize the use of finite-size criteria and crossing estimates as the essential ingredients, but remain model-dependent. In this note, we deduce results of this kind relying solely on a short list of general hypotheses that can be used to recover results for critical percolation on a large class of dependent models.

Consider a two-dimensional percolation model on $\Z^{2}$ indexed by a parameter $\lambda \in [0,1]$ and denote by $\P_{\lambda}$ the distribution of the process at $\lambda$.

We collect a list of assumptions on the model. Our first assumption is that
\begin{equation}\label{h1}
    \tag{H1}
    \text{the model is monotone with respect to $\lambda$}.
\end{equation}
This means that, if $A \subset \{0,1\}^{\Z^{2}}$ is an increasing event, then $\lambda \mapsto \P_{\lambda}(A)$ is a non-decreasing function.

The second assumption regards invariance properties with respect to symmetries of $\Z^{2}$:
\begin{equation}\label{h2}
    \tag{H2}
    \begin{array}{c}
    \text{$\P_\lambda$ is invariant under symmetries of $\Z^{2}$} \\
    \text{and ergodic with respect to shifts.}
    \end{array}
\end{equation}

The following two assumptions control the correlations of the percolation model. The first assumption states that
\begin{equation}\label{h3}
    \tag{H3}
    \text{$\P_{\lambda}$ satisfies the FKG inequality.}
\end{equation}
The second assumption quantifies the decay rate of correlation. There exists a non-increasing sequence $r_{n} \geq 0$ such that, for every two events $A$ and $B$ with supports in $[-3n,3n]^{2}$ and $\Z^{2} \setminus [-4n,4n]^{2}$, respectively, it holds that, for every $\lambda$,
\begin{equation}\label{eq:covariance}
    \Cov_{\lambda} \big( \textbf{1}_{A}, \textbf{1}_{B} \big) \leq r_{n},
\end{equation}
with
\begin{equation}\label{h4}
    \tag{H4}
    \sum_{n=0}^{\infty} \frac{r_{n}}{n} < \infty.
\end{equation}

Hypothesis~\eqref{h4} is a quantitative decoupling assumption that controls the dependence between events supported on distant sets. It replaces independence in the classical Bernoulli setting and is the main hypothesis distinguishing our framework from previous ones.

We now state two final hypotheses concerning the behavior of the probabilities of local events with respect to the parameter $\lambda$. The first one states that, for every finitely-supported event $A$,
\begin{equation}\label{h5}
    \tag{H5}
    \lambda \mapsto \P_{\lambda}(A) \text{ is continuous}.
\end{equation}

The second hypothesis requires that, for every finitely-supported non-trivial increasing event $A$,
\begin{equation}\label{h6}
    \tag{H6}
    \lim_{\lambda \to 0} \P_{\lambda}(A) = 0 \qquad \text{and} \qquad \lim_{\lambda \to 1} \P_{\lambda}(A) = 1.
\end{equation}

\begin{remark}
    Hypothesis~\eqref{h6} requires that the model behaves trivially for extreme values of the parameter, ensuring that the phase transition occurs at a finite parameter (see Theorem~\ref{t:phase_transition} below). I fact, under~\eqref{h2},~\eqref{h6} can be replaced by
\begin{equation}
    \lim_{\lambda \to 0} \P_{\lambda}( \omega_e =1 ) = 0 \qquad \text{and} \qquad \lim_{\lambda \to 1} \P_{\lambda}( \omega_e =1 ) = 1,
\end{equation}
where $\omega_{e}$ denotes the percolation configuration at a given edge $e$.
\end{remark}

\begin{remark}
    Hypotheses~\eqref{h3} and~\eqref{h4} are most restrictive ones, and genuinely of probabilistic nature. On the other hand,~\eqref{h1},~\eqref{h2},~\eqref{h5}, and~\eqref{h6} are regularity properties satisfied by many natural one-parameter percolation models.
\end{remark}

Define the percolation function
\begin{equation}
    \theta(\lambda) = \P_{\lambda} \big( 0 \leftrightarrow \infty \big).
\end{equation}
Under~\eqref{h1}, the function $\theta$ is non-decreasing and one can define the critical percolation threshold
\begin{equation}
    \lambda_{c} = \inf \{ \lambda : \theta(\lambda) >0\}.
\end{equation}

Notice that $\lambda_{c}$ does not need to be well defined in general nor it needs to be non-trivial. Our first result regards the existence of percolation at the critical value $\lambda_{c}$ when both conditions are satisfied.
\begin{theorem}\label{t:no_percolation}
    Given a percolation model satisfying~\eqref{h1},~\eqref{h2},~\eqref{h3},~\eqref{h4},~\eqref{h5}, and such that $\lambda_{c} \in (0,1)$, we have
    \begin{equation}
        \theta(\lambda_{c}) = \P_{\lambda_{c}} \big( 0 \leftrightarrow \infty \big) = 0.
    \end{equation}
\end{theorem}

\begin{remark}\label{rmk:t1}
We list a few observations on the result above.
\begin{enumerate}
    \item Although the theorem is stated for parameters $\lambda \in [0,1]$, one can also consider different parametrizations, such as $\lambda >0$ or $\lambda \in \R$.
    \item The main difficulty in applying the theorem is verifying the quantitative decoupling hypothesis~\eqref{h4}; the remaining assumptions are standard for most positively associated percolation models (see Section~\ref{sec:examples}).
    \item The most restrictive hypothesis in the above is the correlation decay in~\eqref{h4}, as it is required to be uniform over the parameter $\lambda$. This can be weakened by asking it to be uniform over compact intervals in $(0,1)$.
    \item The same proof still holds for site percolation and continuous percolation models in $\R^{2}$.
\end{enumerate}
\end{remark}

The next theorem states that~\eqref{h6} implies that the percolation phase transition is indeed non-trivial.
\begin{theorem}\label{t:phase_transition}
    Assume the percolation model satisfies~\eqref{h1},~\eqref{h2},~\eqref{h3},~\eqref{h4}, and~\eqref{h6}. Then the phase transition is non-trivial, that is, $\lambda_{c} \in (0,1)$.
\end{theorem}

Combining the results above, we immediately obtain
\begin{theorem}\label{t:main}
    Assume the percolation model satisfies~\eqref{h1}--\eqref{h6}. Then the critical threshold $\lambda_{c}$ is non-trivial and $\theta(\lambda_{c}) = 0$.
\end{theorem}
\bigskip

\noindent\textbf{Proof sketches.}
The proof of the facts above are based on establishing finite-size criteria for percolation, in terms of probabilities of box-crossing events:
\begin{equation}\label{eq:crossing}
    \mathcal{C}(n,m) = \big\{\text{there exists an open left-right crossing of the box } [0,n] \times [0,m] \big\}.
\end{equation}

\begin{remark}
    In fact, in the proofs of Theorems~\ref{t:no_percolation} and~\ref{t:phase_transition}, we only use the information in hypotheses~\eqref{h5} and~\eqref{h6} for the crossing events above. In particular both results remain valid if one rephrases~\eqref{h5} and~\eqref{h6} by considering only crossing events.
\end{remark}

The key intermediate steps in the proof of Theorems~\ref{t:no_percolation} and~\ref{t:phase_transition} are the following two propositions.
\nc{c:fsc}
\begin{proposition}\label{prop:fsc}
Under~\eqref{h2},~\eqref{h3}, and~\eqref{h4}, there exists $\uc{c:fsc}>0$ such that, if
\begin{equation}\label{eq:fsc}
    \P_{\lambda} \big( \mathcal{C}(3n,n) \big) >1-\uc{c:fsc}, \text{ for some } n \in \N,
\end{equation}
then $\theta(\lambda) >0$.
\end{proposition}

While versions of Proposition~\eqref{prop:fsc} are widely available for a gamma of different models, the heart of the proof of Theorems~\ref{t:no_percolation} and~\ref{t:phase_transition} relies in its counterpart, stated below.
\begin{proposition}\label{prop:crossing_to_one}
    Assume~\eqref{h2},~\eqref{h3}, and~\eqref{h4}. If $\theta(\lambda)>0$, then
    \begin{equation}
        \lim_{n \to \infty} \P_{\lambda}\big( \mathcal{C}(2n+1,2n+1) \big) = 1.
    \end{equation}
\end{proposition}

When combined, the two propositions above allow one to prove that the finite-size criteria~\eqref{eq:fsc} is a necessary and sufficient condition for percolation to occur (see the proof of Theorem~\ref{t:no_percolation}).  

\section{Proof ingredients}

The key result we use in the proofs is a general RSW estimate developed by K{\"o}hler-Schindler and Tassion~\cite{general_RSW}, stated below.
\begin{theorem}[RSW~\cite{general_RSW}]\label{lemma:rsw}
    Under hypotheses~\eqref{h2} and~\eqref{h3}, there exists a non-decreasing homeomorphism $\psi: [0,1] \to [0,1]$ such that
    \begin{equation}
        \P_{\lambda}\big( \mathcal{C}(3n,n) \big) \geq \psi \big( \P_{\lambda}\big( \mathcal{C}(n,n) \big) \big)
    \end{equation}
\end{theorem}

Furthermore, let us recall the square-root trick, an easy consequence of the FKG inequality (see Grimmett~\cite{grimmett}). Let $A_{1}, \dots, A_{n}$ be increasing events and $A = \cup_{i=1}^{n} A_{i}$. Then
\begin{equation}\label{eq:srt}
    \max_{1 \leq i \leq n} \P_{\lambda}(A_{i}) \geq 1- \big( 1- \P_{\lambda}(A) \big)^{\frac{1}{n}}.
\end{equation}

Finally, Cauchy's condensation test is a classical result regarding convergence of series of non-negative decreasing terms: if $x_{n}$ is a decreasing sequence of positive numbers converging to zero, then, for every integer $k \geq 2$,
\begin{equation}\label{eq:cauchy_condensation}
    \sum_{n=1}^{\infty} x_{n} \text{ converges if, and only if, } \sum_{n=1}^{\infty} k^{n} x_{k^{n}} \text{ converges.}
\end{equation}

\section{Proof of Proposition~\ref{prop:fsc}}

The proof of Proposition~\ref{prop:fsc} is based on a simple renormalization procedure.

Consider the sequence of scales $L_{k}$, by setting $L_{0} = \ell$ and $L_{k+1}=3L_{k}$, so that $L_{k} = 3^{k}\ell$, where $\ell$ will be a large integer chosen later in Lemma~\ref{lemma:recursive_inequality}.

Consider the events $\mathcal{C}(3L_{k}, L_{k})$ and define, for $k \geq 0$,
\begin{equation}
    p_{k} = \P_{\lambda} \big( \mathcal{C}(3L_{k}, L_{k})^{c} \big) = 1- \P_{\lambda} \big( \mathcal{C}(3L_{k}, L_{k})\big).
\end{equation}

\begin{lemma}\label{lemma:recursion}
    For every $k \geq 0$ and any choice of $\ell$,
    \begin{equation}
        p_{k+1} \leq 49(p_{k}^{2}+r_{L_{k}}).
    \end{equation}
\end{lemma}

\begin{proof}
    Start by considering the collection of boxes
    \begin{equation}\label{eq:first_collection}
      \begin{split}
        B_{j} = (2jL_{k},0) + [0,3L_{k}] & \times [0,L_{k}], \text{ for } j = 0,1,2,3 \text{ and } \\
        B_{j} = (2(j-3)L_{k},0) + [0,L_{k}] & \times [-2L_{k}, L_{k}], \text{ for } j=4,5,6.
      \end{split}
    \end{equation}
    If all of the boxes above are crossed in the longer direction, one can concatenate such crossings and obtain a crossing of the larger rectangle $[0, 3L_{k+1}] \times [0, L_{k+1}]$, effectively implying that the event $\mathcal{C}(3L_{k+1}, L_{k+1})$ holds.

    The same reasoning can be applied to the collection of boxes given by
    \begin{equation}
      \begin{split}
       \tilde{B}_{j} = (2jL_{k},0) + [0,3L_{k}] & \times [2L_{k},3L_{k}], \text{ for } j = 0,1,2,3 \text{ and } \\
        \tilde{B}_{j} = (2(j-3)L_{k},0) + [0,L_{k}] & \times [2L_{k}, 5L_{k}], \text{ for } j=4,5,6.
      \end{split}
    \end{equation}

    Second, given a pair of boxes $B_{j}$ and $\tilde{B}_{i}$, the crossing events of these boxes have probability $p_{k}$ and the correlation between them can be bound using \eqref{eq:covariance}, yielding the bound
    \begin{equation}
        \P_{\lambda}( B_{j} \text{ and } \tilde{B}_{i} \text{ are crossed in the hard direction} ) \leq p_{k}^{2} + r_{L_{k}}.
    \end{equation}

    Finally, union bound now gives
    \begin{equation}
        p_{k+1} \leq \sum_{i,j=0}^{6} 
        \P_{\lambda}( B_{j} \text{ and } \tilde{B}_{i} \text{ are crossed in the hard direction} ) \leq 49(p_{k}^{2} + r_{L_{k}}),
    \end{equation}
    concluding the proof.
\end{proof}

\begin{lemma}\label{lemma:recursive_inequality}
    There exists $n_{0}$ such that, if $p_{0} \leq \frac{1}{100}$ for some choice of $\ell \geq n_{0}$, then
    \begin{equation}
        \sum_{k \geq 0} p_{k} < \infty.
    \end{equation}
\end{lemma}

\begin{proof}
    Let $n_{0}$ such that
    \begin{equation}
       r_{n} \leq \frac{1}{10000},
    \end{equation}
    for all $n \geq n_{0}$.

    Assume that $p_{0} \leq \frac{1}{100}$, for some $\ell \geq n_{0}$. Let us prove by induction that $p_{k} \leq \frac{1}{100}$, for all $k \geq 0$. Indeed, if the inequality holds for $k$, Lemma~\ref{lemma:recursion} implies
    \begin{equation}
        p_{k+1} \leq 49 (p_{k}^{2}+r_{L_{k}}) \leq 49 \Big( \frac{1}{10000}+\frac{1}{10000} \Big)  \leq \frac{1}{100}.
    \end{equation}
    
    Combining the above with Lemma~\ref{lemma:recursion} implies
    \begin{equation}
        p_{k+1} \leq \frac{1}{2}p_{k-1} + 49 r_{L_{k}},
    \end{equation}
    for all $k \geq 0$. Inductively applying the last inequality yields, for $i < k$,
    \begin{equation}
        p_{k+1} \leq \frac{1}{2^{i}}p_{k-i} + 49\sum_{j=1}^{i} \frac{1}{2^{j-1}}r_{L_{k-j+1}} \leq \frac{1}{2^{i}} + 98 r_{L_{k-i+1}}.
    \end{equation}
    Choosing $i = \lfloor \frac{k}{2} \rfloor$, we obtain
    \begin{equation}
        \sum_{k=0}^{\infty} p_{k} \leq 2\sum_{i=0}^{\infty} \frac{1}{2^{i}} + 98 \times 2 \sum_{i=0}^{\infty} r_{L_{i}}.
    \end{equation}
    To conclude, notice that the last sum above is finite by Cauchy's condensation test~\eqref{eq:cauchy_condensation} and hypothesis~\eqref{h4} (here we use that $\frac{r_{n}}{n}$ is decreasing).
\end{proof}

We are now in position to prove Proposition~\ref{prop:fsc}.
\begin{proof}[Proof of Proposition~\ref{prop:fsc}]
    Let us first verify that the statement holds if one takes $\uc{c:fsc} = \frac{1}{100}$ and $n \geq n_{0}$ given by Lemma~\ref{lemma:recursive_inequality}. In this case, Lemma~\ref{lemma:recursive_inequality} implies
    \begin{equation}
        \sum_{k=1}^{\infty} 7p_{k} < \infty.
    \end{equation}
    
    Borel-Cantelli lemma applied to the collection of boxes in~\eqref{eq:first_collection}, for all values of $k \geq 1$, implies that all but finitely many of them are crossed. Concatenating such crossings, one constructs an infinite open path, implying $\theta(\lambda)>0$. 

    To conclude the proof, it remains to remove the hypothesis that $n \geq n_{0}$, by possibly decreasing the value of $\uc{c:fsc}$. Indeed, by concatenating at most $2n_{0}$ boxes arranged analogously to~\eqref{eq:first_collection}, it follows that, it follows that, for every $n < n_{0}$,
    \begin{equation}
        \P_{\lambda} \big( \mathcal{C}(3n_{0}, n_{0}) \big) \geq \P_{\lambda} \big( \mathcal{C}(3n, n) \big)^{2n_{0}},
    \end{equation}
    In particular, if $\uc{c:fsc} \leq 1-(1-\frac{1}{100})^{\frac{1}{2n_{0}}}$, then
    \begin{equation}
        \P_{\lambda} \big( \mathcal{C}(3n_{0}, n_{0}) \big) \geq (1-\uc{c:fsc})^{2n_{0}} \geq 1-\frac{1}{100},
    \end{equation}
    from which one can carry the same argument presented previously. The proof is complete by taking $\uc{c:fsc} = \min \big\{1-(1-\frac{1}{100})^{\frac{1}{2n_{0}}}, \frac{1}{100} \big\}$.
\end{proof}

\section{Proof of Proposition~\ref{prop:crossing_to_one}}

\begin{lemma}\label{lemma:lower_bound_box_crossing}
    Assume~\eqref{h2},~\eqref{h3}, and $\theta(\lambda)>0$. Then
    \begin{equation}
        \inf_{n \geq 0} \P_{\lambda}\big( \mathcal{C}(n,n) \big) >0.
    \end{equation}
\end{lemma}

\begin{proof}
    Let us first prove that $\P_{\lambda}\big( \mathcal{C}(2k+1,2k+1) \big)$ is uniformly positive over the choice of $k$, by considering crossings of the box $[-k,k]^{2}$.

    Define by $R_{k}$ (respectively, $T_{k}$, $L_{k}$, and $B_{k}$) the event where the origin is connected by an open path contained in $[-k,k]^{2}$ to the right face (respectively, top, left face, and bottom ) of the box $\{k\} \times [-k,k]$. Notice that rotation invariance implies
    \begin{equation}
        \P_{\lambda}(R_{k}) = \P_{\lambda}(T_{k}) = \P_{\lambda}(L_{k}) = \P_{\lambda}(B_{k}) \geq \frac{\theta(\lambda)}{4}.
    \end{equation}

    We conclude by applying the FKG inequality below
    \begin{equation}
        \P_{\lambda}\big( \mathcal{C}(2k+1, 2k+1) \big) \geq \P_{\lambda} (L_{k} \cap R_{k}) \geq \frac{\theta(\lambda)^{2}}{16},
    \end{equation}
    implying the statement for odd values of $n$.

    As for even values, simply notice that we can obtain a crossing of $[0,2k]^{2}$ by first crossing $[0,2k-1]^{2}$ and then opening a final extra edge. 
\end{proof}

For $m < n$, denote by $\Circ(m,n)$ the event
\begin{equation}\label{eq:circ}
    \Circ(m,n) =
    \bigg\{
      \begin{array}{c}
        \text{there exists an open circuit surrounding } [-m.m]^{2} \\
        \text{entirely contained in } [-n,n]^{2} 
      \end{array}
    \bigg\}.
\end{equation}

\nc{c:circ}
\begin{lemma}\label{lemma:circ}
    Assume~\eqref{h2},~\eqref{h3},~\eqref{h4}, and $\theta(\lambda)>0$. There exist $\uc{c:circ}>0$ and $\alpha>0$ such that, for any $m \leq n$,
    \begin{equation}
        \P_{\lambda}\big( \Circ(m,n)^{c} \big) \leq \uc{c:circ} \Big( \Big(\frac{n}{m} \Big)^{-\alpha} + r_{m}\log \frac{n}{m} \Big).
    \end{equation}
    In particular,
    \begin{equation}\label{eq:circ_limit}
        \lim_{k \to \infty} \P_{\lambda} \big( \Circ( \lfloor \sqrt{k} \rfloor, 2k+1)^{c} \big) = 0.
    \end{equation}
\end{lemma}

\begin{proof}
    The proof is split into some steps. First, notice that Lemma~\ref{lemma:lower_bound_box_crossing} combined with the RSW estimate from Lemma~\ref{lemma:rsw} yields
    \begin{equation}
        c := \inf_{n} \P_{\lambda}\big( \mathcal{C}(3n,n) \big) >0 
    \end{equation}

    Combining four crossings of rectangles in the hard direction implies
    \begin{equation}
        \P_{\lambda}\big( \Circ(k, 3k) \big) \geq c^{4},
    \end{equation}
    for every $k \geq 0$.

    Given now $m<n$, consider the sequence $k_{i}$, for $i=1, \dots, \ell$, given by $k_{1} = m$, $k_{i+1} = 9k_{i}$ and $\ell = \lfloor \log_{9} \frac{n}{3m} \rfloor$. In particular, $3k_{\ell} \leq n$ and thus
    \begin{equation}
        \P_{\lambda} \big( \Circ(m,n)^{c} \big) \leq \P_{\lambda} \Big( \bigcap_{i=1}^{\ell} \Circ(k_{i}, 3k_{i})^{c} \Big).
    \end{equation}

    We now use the estimate above together with the correlation decay estimate in~\eqref{eq:covariance}. Notice that, while the event $\cap_{i=1}^{j} \Circ(k_{i}, 3k_{i})^{c}$ is supported in $[-3k_{j}, 3k_{j}]^{2}$, $\bigcap_{i=j+1}^{\ell} \Circ(k_{i}, 3k_{i})^{c}$ has support in $\Z^{2} \setminus [-9k_{j}, 9k_{j}]^{2}$. In particular, repeated applications of~\eqref{eq:covariance} yields
    \begin{equation}
      \begin{split}
         \P_{\lambda} \big( \Circ(m,n)^{c} \big) & \leq (1-c^{4})^{\ell}+\sum_{i=1}^{\ell-1} r_{k_{i}} \\
         & \leq (1-c^{4})^{\lfloor \log_{9} \frac{n}{3m} \rfloor} + r_{m} \log_{9} \frac{n}{3m} \\
         & \leq \uc{c:circ} \Big( \Big(\frac{n}{m} \Big)^{-\alpha} + r_{m} \log \frac{n}{m} \Big),
      \end{split}
    \end{equation}
    concluding the proof of the first bound.
    
    As for~\eqref{eq:circ_limit},
    \begin{equation}
       \P_{\lambda} \big( \Circ( \lfloor \sqrt{k} \rfloor, 2k+1)^{c} \big) \leq \uc{c:circ} \bigg( \bigg(\frac{2k+1}{\lfloor \sqrt{k} \rfloor} \bigg)^{-\alpha} + r_{\lfloor \sqrt{k} \rfloor}\log \frac{2k+1}{\lfloor \sqrt{k} \rfloor} \bigg).
    \end{equation}
    To conclude, we just need to verify that~\eqref{h4} implies $r_{n} \log n \to 0$ as $n$ grows. Indeed, notice first that Cauchy's condensation test~\eqref{eq:cauchy_condensation} yields
    \begin{equation}
        \sum_{k=0}^{\infty} r_{2^{k}} = \sum_{k=0}^{\infty} 2^{k} \frac{r_{2^{k}}}{2^{k}}  < \infty.
    \end{equation}
    Now, since $r_n$ is decreasing,
    \begin{equation}
        \frac{j}{2}r_{2^{j}} \leq \sum_{k \geq \frac{j}{2}} r_{2^{k}} \to 0,
    \end{equation}
    as $j$ grows, proving the statement for the subsequence $r_{2^{j}} \log 2^{j}$. As for the whole sequence, given $n$, choose $j$ such that $2^{j} \leq n < 2^{j+1}$ and notice that
    \begin{equation}
        r_{n} \log n \leq r_{2^{j}} \log 2^{j+1} \to 0,
    \end{equation}
    concluding the proof.
\end{proof}

 We now conclude the proof of Proposition~\ref{prop:crossing_to_one}.
\begin{proof}[Proof of Proposition~\ref{prop:crossing_to_one}]
    We will prove that the probability that the square $[-n,n]^{2}$ is crossed goes to one as $k$ increases.

    First, let $A_{n}$ denote the event that there exists an infinite connected component that touches the box $[-\lfloor \sqrt{n} \rfloor, \lfloor \sqrt{n} \rfloor ]$. Similarly to Lemma~\ref{lemma:lower_bound_box_crossing}, define the events $R_{n}$, $T_{n}$, $L_{n}$, and $B_{n}$ where this infinite component leaves $[-n,n]^{2}$ by the right, top, left, and bottom faces, respectively. Once again, by rotation invariance, all these events have positive probability and $A_{n}$ is the union of these four events. In particular, the square root trick~\eqref{eq:srt} implies that
    \begin{equation}\label{eq:right_connection_to_one}
        \P_{\lambda}\big(R_{n} \big) \geq 1 - \big( 1-\P_{\lambda}(A_{n}) \big)^{\frac{1}{4}} \to 1,
    \end{equation}
    since $\P_{\lambda}(A_{n}) \to 1$, by the ergodicity assumption~\eqref{h2}.

    Notice now that
    \begin{equation}
      \begin{split}
        \P_{\lambda}\big( \mathcal{C}(2n+1,2n+1) \big) & \geq \P_{\lambda} \big( \Circ( \lfloor \sqrt{n} \rfloor, n) \cap R_{n} \cap L_{n} \big) \\
        & \geq  \Big( 1- \P_{\lambda} \big( \Circ( \lfloor \sqrt{n} \rfloor, n)^{c} \big) \Big) \P_{\lambda}\big( R_{n} \big) \P_{\lambda}\big( L_{n} \big).
      \end{split}
    \end{equation}
    The proof is now complete by combining~\eqref{eq:right_connection_to_one} and Lemma~\ref{lemma:circ}.    
\end{proof}

\section{Proof of main results}

We start with the proof of Theorem~\ref{t:no_percolation}.
\begin{proof}[Proof of Theorem~\ref{t:no_percolation}]
    Assume by contradiction that $\theta(\lambda_{c})>0$. Proposition~\ref{prop:crossing_to_one} implies that there exists $n$ large enough such that
    \begin{equation}
        \P_{\lambda_{c}}\big( \mathcal{C}(2n+1,2n+1) \big) \geq 1-\frac{\uc{c:fsc}}{2}.
    \end{equation}
    Since $\lambda \mapsto \P_{\lambda}\big( \mathcal{C}(2n+1,2n+1)$ is continuous by~\eqref{h5}, it follows that
    \begin{equation}
        \P_{\lambda_{c} - \delta}\big( \mathcal{C}(2n+1,2n+1) \big) \geq 1- \uc{c:fsc},
    \end{equation}
    for $\delta>0$ sufficiently small. Proposition~\ref{prop:fsc} now implies that $\theta(\lambda_{c}-\delta)>0$, yielding a contradiction.
\end{proof}

The proof of Theorem~\ref{t:phase_transition} relies on duality. Consider the dual graph $(\Z^{2})^{*} = \Z^{2} + (\frac{1}{2}, \frac{1}{2})$ with edges between nearest neighbors. Each edge $e \in \Z^{2}$ is crossed by an unique edge $e^{*} \in (\Z^{2})^{*}$ and vice-versa. Given a configuration $\omega$, define the dual configuration $\omega^{*}$ in $ \Z^{2} + (\frac{1}{2}, \frac{1}{2})$ by setting
\begin{equation}
    \omega^{*}(e^{*}) = 1-\omega(e).
\end{equation}

Analogously to~\eqref{eq:crossing}, define the dual vertical crossing of a box
\begin{equation}
    \mathcal{C}^{*}(n,m) = \bigg\{\begin{array}{c}
      \text{there exists a dual open path crossing of} \\
      \text{$[0,n] \times [0,m]$ from top to bottom}
    \end{array}\bigg\}.
\end{equation}
Observe that~\eqref{h2},~\eqref{h3}, and~\eqref{h4} still hold for the dual configurations.

Furthermore, for all $m,n \in \N$,
\begin{equation}\label{eq:complement}
    \mathcal{C}(n,m) = \mathcal{C}(n,m)^{c}.
\end{equation}

Let us now prove Theorem~\ref{t:phase_transition}.
\begin{proof}[Proof of Theorem~\ref{t:phase_transition}]
    Observe first that~\eqref{h6} applied to $A = \mathcal{C}(3n,n)$ immediately implies that the condition of Proposition~\ref{prop:fsc} is verified and thus $\lambda_{c}$ is well defined and strictly smaller than one.

    It remains to prove that $\lambda_{c}>0$, which will follow once we prove that there is no percolation for small enough $\lambda$.

    We will proceed in two steps. First, we prove that, if $\lambda$ is small enough, the dual configuration percolates. In fact, this follows from another application of Lemma~\ref{prop:fsc}. By~\eqref{eq:complement}, we have
    \begin{equation}
        \lim_{\lambda \to 0} \P_{\lambda} \big( \mathcal{C}*(n,3n) \big) = \lim_{\lambda \to 0} 1 - \P_{\lambda} \big( \mathcal{C}(n,3n) \big) = 1.
    \end{equation}
    In particular, Proposition~\ref{prop:fsc} implies that the dual configuration percolates of $\lambda$ small enough.

    Finally, let us prove that there is no primal percolation if the the dual percolates. Notice first that, for all $n \in \N$,
    \begin{equation}
        \theta(\lambda) \leq \P_{\lambda} \big( \Circ^{*}( \lfloor \sqrt{n} \rfloor, 2n+1)^{c} \big),
    \end{equation}
    where $\Circ^{*}(n,m)$ stands for the analogous of the event~\eqref{eq:circ} for the dual configuration. Since there is percolation of the dual lattice for $\lambda$ small enough, Lemma~\ref{lemma:circ} implies that the right-hand side of the equation above converges to zero as $n$ grows, implying $\theta(\lambda)=0$ and thus $\lambda_{c}>0$. This concludes the proof.
\end{proof}

\section{Examples}\label{sec:examples}

We now verify that the hypotheses of the main theorem are satisfied by some families of dependent percolation models, recovering results available in the literature.

\subsection{Gaussian level-field percolation}

\nc{c:corr}
\par Fix a function $q:\Z^{2} \to \R_{+}$ not identically null, invariant under reflections and rotations of the plane.
Assume further that there exists $\beta > 2$ such that
\begin{equation}\label{eq:correlation_q}
q(x) \leq \uc{c:corr}|x|^{-\beta}, \text{ for all } x \in \Z^{2}\setminus{\{o\}},
\end{equation}
and that $q(0) \leq \uc{c:corr}$, for some $\uc{c:corr}>0$.
Consider now a family $\big(W_{x}\big)_{x \in \Z^{2}}$ of i.i.d.\ standard $\mathcal{N}(0,1)$ random variables and define the random Gaussian field $( g_{x} )_{x \in \Z^{2}}$ via
\begin{equation}\label{eq:construction_gf}
g_{x} = \sum_{y}q(x-y)W_{y}.
\end{equation}

\begin{remark}
Sharpness of phase transition for these models was established by Muirhead~\cite{m}. As we prove in the following, Theorems~\ref{t:no_percolation} and~\ref{t:phase_transition} can be applied if $\beta> 3$, partially recovering Muirhead's result.
\end{remark}

Given the Gaussian field $g$, and $\lambda \in \R$, consider the percolation configuration given by
\begin{equation}
\label{eq:field}
\omega_{\lambda}(x) = \textbf{1}_{\{g_{x} \leq \lambda \}}
\end{equation}
Denote by $\P_{\lambda}$ the distribution of $\omega_{\lambda}$ and observe that it trivially satisfies ~\eqref{h1},~\eqref{h5}, and~\eqref{h6}. The only non-trivial assumptions are~\eqref{h2},~\eqref{h3}, and~\eqref{h4}. While the FKG inequality follows from the fact that the field we consider is positively correlated,~\eqref{h3} and~\eqref{h4} actually rely on hypotheses on the decay of the covariances of the field.

The FKG inequality follows from the following lemma, establishing~\eqref{h3}.
\begin{proposition}[FKG inequality for Gaussian fields~\cite{pitt1982}]
  \label{prop:FKG}
  Let $A$ and $B$ be two increasing events depending on finitely many coordinates of $g$. Then
  \begin{equation}\label{e:fkg_gf}
    \tag{FKG}
    \P \big( g \in A \cap B\big) \geq \P \big( g \in A \big) \P \big( g \in B \big).
  \end{equation}
\end{proposition}

For every $x \in \Z^{2}$, $g_{x}$ is a centered Gaussian random variable with variance
\begin{equation}
\label{eq:variance}
\sigma_{q}^2=\sum_{y \in \Z^{2}} q(y)^{2} < \infty,
\end{equation}
since $\beta >2$.
The covariance structure in this field is given by
\begin{equation}
\label{e:covariance_g}
\Cov\big(g_{x},g_{y}\big) = \sum_{z \in \Z^{2}}q(x-z)q(y-z) = \sum_{z \in \Z^{2}}q(x-y-z)q(z) = q*q(x-y).
\end{equation}

\nc{c:correlation_convolution}
Due to~\eqref{eq:correlation_q}, we have
\begin{equation}
\begin{split}
  q*q(x) & = \sum_{z: \, |z-x| \geq \frac{1}{2}|x|}q(x-z)q(z) + \sum_{z: \, |z-x| < \frac{1}{2}|x|}q(x-z)q(z) \\
  & \leq 2\sum_{y: \, |y| \geq \frac{1}{2}|x|}\uc{c:corr}^{2}|y|^{-\beta} \leq \uc{c:correlation_convolution}|x|^{-\beta+2},
\end{split}
\end{equation}
for some constant $\uc{c:correlation_convolution} > 0$.

In order to establish~\eqref{h3}, we first notice that invariance with respects to symmetries of $\Z^{2}$ follows directly from the definition~\eqref{eq:construction_gf}. Ergodicity follows from the following lemma.
\begin{lemma}
    If $\beta>2$, the Gaussian field $(g_{x})_{x \in \Z^{2}}$ is strong mixing, and thus, ergodic with respect to shifts.
\end{lemma}

\begin{proof}
    Let $A$ and $B$ be finitely-supported events depending only on the entries of $g$ in $[-L,L]^{2}$, for some $L>0$. Denote by $\uptau_{z}: \R^{\Z^{2}} \to \R^{\Z^{2}}$ the shift by $z \in \Z^{2} \setminus \{0\}$. 

    Our goal is to prove that
    \begin{equation}
        \lim_{n} \P ( g \in A \cap \uptau^{-1}_{nz}(B) ) = 0.
    \end{equation}

    Notice that $\uptau^{-1}_{nz}(B)$ is an event that depends only on the variables $g_{y}$, for $y \in -nz+[-L, L]^{2}$.
    
    Since $\beta>2$, it follows
\begin{equation}
    \Cov\big(g_{x},g_{y}\big) \to 0, \text{ as } |x-y| \to +\infty.
\end{equation}

    Furthermore, the covariance matrix $\Gamma^{n}$ between the variables $(g_{x})_{x \in [-L,L]^{2}}$ and $(g_{y})_{y \in -nz+[-L, L]^{2}}$ satisfies
    \begin{equation}
        \Gamma^{n}(x_{1}, x_{2}) = \Cov(g_{x_{1}}, g_{x_{2}}), \text{ if } x_{1}, x_{2} \in [-L,L]^{2} \text{ or } x_{1}, x_{2} \in -nz+[-L, L]^{2}
    \end{equation}

    As the weak limit of Gaussian random variables is Gaussian with limit means and variances, it follows that the vector $\big( (g_{x})_{x \in [-L,L]^{2}}, (g_{y})_{y \in -nz+[-L, L]^{2}} \big)$ converges weakly to $\big( (g_{x})_{x \in [-L,L]^{2}}, (g'_{x})_{x \in [-L,L]^{2}} \big)$, where $(g'_{x})_{x \in [-L,L]^{2}}$ is an independent copy of $(g_{x})_{x \in [-L,L]^{2}}$.

    This implies the result if $A$ and $B$ are cylinders. The general case follows from standard measure-theoretical arguments.
\end{proof}

Hypothesis~\eqref{h4} is the one that demands the most work to verify. The proof we present here is based on~\cite[Section~6.1]{bhkt} and relies on constructing finite-range approximations for $g$. For each $r>0$, define the truncated filed $\big(g_{x}^{r}\big)_{x \in \Z^{2}}$ via
\begin{equation}\label{eq:finite_range_approximation}
g^{r}_{x} = \sum_{y: \, |x-y| \leq \frac{r}{2}}q(x-y)W_{y}.
\end{equation}
Notice that, if $|x-y|>r$, then $g_{x}^{r}$ and $g_{y}^{r}$ are independent.

\nc{c:approximation}
\begin{lemma}\label{lemma:finite_range_approximation}
There exists a positive constant $\uc{c:approximation}>0$ such that, for every $r \geq 2$, $x \in \Z^{2}$, and $t \geq 0$,
\begin{equation}
\P \big( |g_{x}-g_{x}^{r}| \geq t \big) \leq 2e^{-\uc{c:approximation}t^{2}r^{2\beta-2}}.
\end{equation}
\end{lemma}

\begin{proof}
We consider only the case $x=0$, the others following from translation invariance. Observe that
\begin{equation}
g_{0}-g_{0}^{r} = \sum_{y: \, |y| > \frac{r}{2}}q(-y)W_{y}
\end{equation}
is a centered Gaussian random variable with variance $\sum_{y: \, |y| > \frac{r}{2}}q(y)^{2}$.
Standard tail estimates yield
\begin{equation}
\P \big( |g_{0}-g_{0}^{r}| \geq t \big) \leq 2e^{-\frac{t^{2}}{2 \sum_{y: \, |y| > \frac{r}{2}}q(y)^{2}}}.
\end{equation}
To conclude, we estimate
\begin{equation}
\sum_{y: \, |y| > \frac{r}{2}}q(y)^{2} \leq \uc{c:corr}^{2}\sum_{y: \, |y| > \frac{r}{2}}|y|^{-2\beta} \leq 4\uc{c:corr}^{2}\sum_{n>\frac{r}{2}}n^{-2\beta+1} \leq \frac{2\uc{c:corr}^{2}}{\beta-1} \Big( \frac{r}{2}-1 \Big)^{-2\beta+2},
\end{equation}
finishing the proof of the proposition.
\end{proof}

Analogously to~\eqref{eq:finite_range_approximation}, we introduce the finite-range approximations of the percolation model, for every $\lambda \in \R$ and $r \geq 1$,
\begin{equation}
\label{eq:field_approximation}
\omega_{\lambda}^{r}(x) = \textbf{1}_{\{g_{x}^{r} \leq \lambda\}}.
\end{equation}

\nc{c:approximation_environment}
We now prove that the finite-range approximations above are unlikely to differ from $\omega_{\lambda}$.
\begin{lemma}
\label{lemma:finite_range_approximation_2}
Given $\varepsilon>0$< there exists $\uc{c:approximation_environment}>0$ such that, for every $\lambda \in \R$, $r \geq 2$, and $x \in \Z^{2}$,
\begin{equation}
\P \big( \omega_{\lambda}(x) \neq \omega_{\lambda}^{r}(x) \big) \leq \uc{c:approximation_environment}r^{-\beta+1+\varepsilon}.
\end{equation}
\end{lemma}

\begin{proof}
Once again, it suffices to consider the case $x=0$. If $\omega_{\lambda}(0) \neq \omega_{\lambda}^{r}(0)$, then for every fixed $t \geq 0$, either $|g_{0}-g_{0}^{r}| \geq t$ or $g_{0} \in (\lambda - t, \lambda + t)$
This yields
\begin{equation}
\P \big( \omega_{\lambda}(0) \neq \omega_{\lambda}^{r}(0) \big) \leq \P \big( |g_{0}-g_{0}^{r}| \geq t \big) + \P\big( g_{0} \in (\lambda-t, \lambda+t) \big).
\end{equation}
Lemma~\ref{lemma:finite_range_approximation} bounds the first term in the right-hand side above. For the second term, notice that $g_{0} \sim \mathcal{N}(0, \sigma_{q})$, with $\sigma_{q}$ as in~\eqref{eq:variance}, and then bound the density of $g_{0}$ by $(2\pi\sigma_{q})^{-\frac{1}{2}}$ to obtain
\begin{equation}
\P\big( g_{0} \in (\lambda-t, \lambda+t) \big) \leq t\sigma_{q}^{-\frac{1}{2}},
\end{equation}
uniformly over $\lambda \in \R$. To conclude the proof, simply choose $t=r^{-\beta+1+\varepsilon}$.
\end{proof}

We are now in position to conclude the verification of~\eqref{h4}.
\nc{c:correlation_decay_gf}
\begin{proposition}\label{prop:verify_h4_gf}
    Assume $\beta>3$. Given $\varepsilon \in (0, \beta-3)$, there exists $\uc{c:correlation_decay_gf}>0$ such that, for all $n \in \N$, if $A$ and $B$ are events supported in $B(0, 3n)$ and $B(0, 4n)^{c}$, respectively, then
    \begin{equation}\label{eq:correlation_decay_gf}
        \Cov_{\lambda} \big( \textnormal{\textbf{1}}_{A}, \textnormal{\textbf{1}}_{B}) \leq \uc{c:correlation_decay_gf} n^{-\beta+3+\varepsilon}.
    \end{equation}
    In particular,~\eqref{h4} is satisfied if $\beta> 3$.
\end{proposition}

\begin{proof}

Define the intermediate field
\begin{equation}
\omega_{\lambda, n}(x) =
  \begin{cases}
    \omega_{\lambda}^{n}(x), & \text{ if } |x| < 4n; \\
    \omega_{\lambda}^{|x|-3n-1}(x), & \text{ if } |x| \geq 4n.
  \end{cases}
\end{equation}

It follows from the definition of the approximated field $g^{r}$ in~\eqref{eq:field_approximation} that the event $\{\omega_{\lambda, n} \in A\}$ depends only on $\{ W_{y}, |y| \leq \frac{3}{2}n\}$ and $\{\omega_{\lambda,n} \in B\}$ is determined by $\{W_{y}: |y| > 3/2n\}$, implying that these two events are independent. We now estimate
\begin{equation}\label{eq:corr_1}
  \begin{split}
      \P_{\lambda} (\omega_{\lambda} \in A \cap B) +  & \leq \P_{\lambda} (\omega_{\lambda,n} \in A \cap B) + \P_{\lambda} ( \omega_{\lambda} \neq \omega_{\lambda,n}) \\
      & = \P_{\lambda} (\omega_{\lambda,n} \in A) \P_{\lambda} (\omega_{\lambda,n} \in B) + \P_{\lambda} ( \omega_{\lambda} \neq \omega_{\lambda,n}) \\
      & \leq \P_{\lambda} (\omega_{\lambda} \in A) \P_{\lambda} (\omega_{\lambda} \in B) + 3\P_{\lambda} ( \omega_{\lambda} \neq \omega_{\lambda,n})
  \end{split}
\end{equation}

In an analogous fashion,
\begin{equation}\label{eq:corr_2}
    \P_{\lambda} (\omega_{\lambda} \in A) \P_{\lambda} (\omega_{\lambda} \in B) \leq \P_{\lambda} (\omega_{\lambda} \in A \cap B) + 3\P_{\lambda} ( \omega_{\lambda} \neq \omega_{\lambda,n}).
\end{equation}
In particular,
\begin{equation}
     \Cov_{\lambda} \big( \textbf{1}_{A}, \textbf{1}_{B}) \leq 3\P_{\lambda} ( \omega_{\lambda} \neq \omega_{\lambda,n}).
\end{equation}

To conclude, we use Lemma~\ref{lemma:finite_range_approximation_2} to bound the probability above:
\begin{equation}
  \begin{split}
    \P_{\lambda} ( \omega_{\lambda} \neq \omega_{\lambda,n}) & \leq \sum_{x \in \Z^{2}} \P_{\lambda} ( \omega_{\lambda} (x) \neq \omega_{\lambda,n}(x)) \\
    & \leq 16 \uc{c:approximation_environment} n^{2} n^{-\beta+1+\varepsilon} + 8 \uc{c:approximation_environment} \sum_{k=4n}^{+\infty} k (k-3n-1)^{-\beta+1+\varepsilon} \\
    & \leq 16 \uc{c:approximation_environment} n^{-\beta+3+\varepsilon} + 8 \uc{c:approximation_environment} \sum_{k=n-1}^{\infty} (k+3n+1) k^{-\beta+1+\varepsilon} \\
    & \leq 16 \uc{c:approximation_environment} n^{-\beta+3+\varepsilon} + 8 \uc{c:approximation_environment} \sum_{k=n-1}^{\infty} 5k \times k^{-\beta+1+\varepsilon} \\
    & \leq 16 \uc{c:approximation_environment} n^{-\beta+3+\varepsilon} + 40 \uc{c:approximation_environment} \frac{1}{\beta-3-\varepsilon}(n-1)^{-\beta+3+\varepsilon}
  \end{split}
\end{equation}
where in the third inequality we choose $n \geq 5$. The proof of~\eqref{eq:correlation_decay_gf} is complete by adjusting the constants.

To verify~\eqref{h4}, notice that $-\beta+2+\varepsilon<-1$ and thus
\begin{equation}
    \sum_{n=1}^{\infty} \frac{r_{n}}{n} \leq \uc{c:correlation_decay_gf} \sum_{n=1}^{\infty} n^{-\beta+2-\varepsilon} < \infty.
\end{equation}
This concludes the proof.
\end{proof}

\subsection{Boolean percolation}

Consider a Poisson point process of points in $\R^{2}$ with intensity $\lambda>0$. Given a realization $\omega$, independently place a circle of random radius $R$ around each point of $\omega$ in an independent fashion.

Let $\mathcal{O}$ denote the region covered by the circles and let $\theta(\lambda)$ be the probability that the origin is covered by a connected component of $\mathcal{O}$ composed of infinitely many circles.

The model satisfies~\eqref{h1},~\eqref{h2},~\eqref{h3},~\eqref{h5}, and~\eqref{h6} (see for example~\cite{mr}). The proposition below verifies~\eqref{h4}.
\begin{proposition}\label{prop:verify_h4}
    If the radius distribution satisfies $\E[ R^{2} \log R]< \infty$ then~\eqref{h4} is satisfied for compact intervals of parameters (see Remark~\ref{rmk:t1}).
\end{proposition}

\begin{proof}
    Fix $n$ and consider the event
  \begin{equation}
    \mathcal{A}_{n} = \big\{ \text{no disk touches } [-3n, 3n]^{2} \text{ and } \R^{2} \setminus [-4n,4n]^{2} \big\},
  \end{equation}

  An analogous computation as in~\eqref{eq:corr_1} and~\eqref{eq:corr_2} implies that, if $A$ and $B$ are events depending only on $[-3n,3n]^{2}$ and $\R^{2} \setminus [-4n,4n]^{2}$, respectively, then
  \begin{equation}
      \Cov_{\lambda}(\textbf{1}_{A}, \textbf{1}_{B}) \leq 3 \P_{\lambda}\big( \mathcal{A}_{n}^{c} \big).
  \end{equation}

  Let us now compute the probability above. Notice that, integrating over the Poisson point process we have
  \begin{equation}
    \begin{split}
      \P_{\lambda}\big( \mathcal{A}_{n}^{c} \big) & \leq \lambda(8n)^{2} \P \Big(R \geq \frac{n}{2} \Big) + \lambda \int_{4n}^{+\infty} 8y \P \big( R> y-3n) \d y \\
      & = 64\lambda n^{2} \P \Big(R \geq \frac{n}{2} \Big) + 8\lambda \int_{n}^{+\infty} (y+3n) \P \big( R > y) \d y \\
      & \leq 64\lambda n^{2} \P \Big(R \geq \frac{n}{2} \Big) + 8\lambda \E \big[ (R^{2}+3nR) \textbf{1}_{R \geq n} \big] \\
    & \leq 64\lambda n^{2} \P \Big(R \geq \frac{n}{2} \Big) + 32\lambda \E \big[R^{2} \textbf{1}_{R \geq n} \big].
    \end{split}
  \end{equation}

  In order to check~\eqref{h4}, we now bound
  \begin{equation}
    \begin{split}
       \sum_{n=1}^{\infty} \frac{r_{n}}{n} & = \sum_{n=1}^{\infty} \frac{3\P_{\lambda}\big( \mathcal{A}_{n}^{c} \big)}{n} \\
       & \leq 192\lambda \sum_{n=1}^{\infty} n \P \Big(R \geq \frac{n}{2} \Big) + 96\lambda \E \Big[ R^{2} \sum_{n=1}^{+\infty} \frac{1}{n} \textbf{1}_{R \geq n} \Big] \\
       & \leq 192\lambda \E \big[ 4R^{2} \big] + 96\lambda \E \Big[ R^{2} (\ln \lceil R \rceil +1) \Big] < \infty,
    \end{split}
  \end{equation}
  by the moment hypothesis on the radius distribution. This concludes the proof.
\end{proof}

\begin{remark}
    While Theorem~\ref{t:no_percolation} implies no percolation at criticality for Boolean percolation when $\E[R^{2} \log R]< \infty$, recovering the result from ~\cite{att}, it also provides an example that~\eqref{h4} is not optimal in all contexts, as~\cite{att2} proves a sharp result: no percolation at criticality actually holds whenever $\E[R^{2}] < \infty$.
\end{remark}

\subsection{Ellipses percolation}

Ellipses percolation is a model introduced by Teixeira and Ungaretti in~\cite{tu}. This model shares some similarities with Booelan percolation: instead of selecting centers for disks via a Poisson point process, here in each point an ellipse with minor axis equal to one and random major axis is placed. 

In order to define the model properly, fix the major axis distribution $\rho$ supported in $[1, +\infty)$ satisfying
\begin{equation}
   cr^{-\alpha} \leq \rho[r, +\infty) \leq \tilde{c}r^{-\alpha},
\end{equation}
for some constants $c,\tilde{c}>0$ and some $\alpha >0$.

Consider a Poisson point process with intensity $\lambda$. Independently place at each point of the process an ellipse with an uniformly chosen direction, minor axis one, and major axis sampled according to $\rho$.

In~\cite{tu}, the authors prove that this model satisfies~\eqref{h1}, \eqref{h2},~\eqref{h3}, and~\eqref{h5}. Furthermore, the same paper proves that the phase transition is non-trivial whenever $\alpha>2$. As for~\eqref{h4}, the proof of Proposition~\ref{prop:verify_h4} still applies in this case. In particular, if $\alpha>2$,
\begin{equation}
  \begin{split}
    \E[R^{2} \log R] & = \int_{1}^{\infty} (2r \log r + r) \rho[r, +\infty) \d r \leq \tilde{c} \int_{1}^{\infty} (2r \log r + r) r^{-\alpha} \d r < \infty.
  \end{split}
\end{equation}

We conclude that, as long as $\alpha>2$, ellipses percolation does not percolate at criticality.

\bibliographystyle{plain}
\bibliography{all}

\end{document}